\documentclass[12pt,reqno]{amsart}
\usepackage{amsmath}
\usepackage{latexsym,amssymb}
\usepackage{mathrsfs}
\usepackage{enumerate}
\usepackage{bm}
\usepackage[all]{xy}
\allowdisplaybreaks[1]

\newtheorem{thm}{Theorem}[section]
\newtheorem{lem}[thm]{Lemma}

\theoremstyle{definition}

\newtheorem{defn}[thm]{Definition}
\newtheorem{rem}[thm]{Remark}
\newtheorem{exam}[thm]{Example}

\numberwithin{equation}{section}

\newcommand{\N}{\mathbb{N}}
\newcommand{\Z}{\mathbb{Z}}

\newcommand{\R}{\mathbb{R}}
\newcommand{\C}{\mathbb{C}}

\def\sI{\mathscr{I}}
\def\sE{\mathscr{E}}
\def\sM{\mathscr{M}}

\DeclareMathOperator{\Ad}{Ad}
\DeclareMathOperator{\Aut}{Aut}

\DeclareMathOperator{\id}{id}

\def\al{\alpha}
\def\hal{{\widehat{\al}}}

\def\be{\beta}
\def\hbe{{\widehat{\be}}}

\def\de{\delta}

\def\la{\lambda}

\def\vep{\varepsilon}

\def\ps{{\psi}}

\def\vph{\varphi}

\def\om{\omega}
\def\si{\sigma}

\def\th{\theta}

\def\De{\Delta}
\def\Ga{\Gamma}
\def\La{\Lambda}

\def\col{\colon}

\def\subs{\subset}
\def\ovl{\overline}
\def\oti{\otimes}
\def\rti{\rtimes}

\begin{document}
\title{Ultraproducts of crossed product von Neumann algebras}

\author[R. Tomatsu]{Reiji Tomatsu}
\address{
Department of Mathematics, Hokkaido University,
Hokkaido \mbox{060-0810},
JAPAN}
\email{tomatsu@math.sci.hokudai.ac.jp}

\subjclass[2010]{Primary 46L10; Secondary 46L40}
\keywords{von Neumann algebra}
\thanks{The author was partially supported by JSPS KAKENHI Grant Number 15K04889.}

\maketitle

\begin{abstract}
We study a relationship between the ultraproduct of
a crossed product von Neumann algebra
and
the crossed product of an ultraproduct von Neumann algebra.
As an application, the continuous
core of an ultraproduct von Neumann algebra is described.
\end{abstract}

\section{Introduction}
Ultraproduct technique
is utilized by many researchers
for settling problems
concerned with
classification theory of C$^*$-algebras
or von Neumann algebras.
Among them,
Connes' study of automorphism analysis is remarkable.
He used central sequence von Neumann algebras
in an effective way to classify outer automorphisms
on an injective type II$_1$ factor.
Since then, classification of group actions
on injective factors has been one of the main topics
in theory of operator algebra,
and
we now have a complete classification for amenable discrete
group actions on them \cite{Co-outer,J-act,Kat-S-T,KawST,Oc}.
See \cite{M} for a unified and direct proof.

For continuous groups such as $\R$,
however, we have to say classification of their actions
on injective factors
is far from being complete.
The main difficulty is that continuous group actions do not
continuously extend to ultraproduct von Neumann algebras.
This leads us to the notion of equicontinuity of norm bounded
sequences that is introduced by Kishimoto \cite{Kishi-CMP}.
Then an equicontinuous part of an ultraproduct von Neumann algebra
could be small in general, but this could provide us
with remedy for study of continuous group actions.

The aim of this paper is to describe a crossed products of
an equicontinuous part of an ultraproduct von Neumann algebra.
The main theorem states the crossed product
coincides with
the equicontinuous part of the ultraproduct of the crossed product
von Neumann algebra with respect to the dual action.
Then we obtain the continuous crossed product decomposition
of a type III ultraproduct von Neumann algebra.
This decomposition
also gives us a characterization of fullness of
a type III$_1$ factor in terms of its continuous core.

\vspace{10pt}
\noindent
{\bf Acknowledgements.}
The author is grateful
to Yoshimichi Ueda
for various advice and valuable comments on this paper.

\section{Preliminary}

\subsection{Ultraproduct}
Our references are \cite{AH,Oc}.
In this paper, we denote by $\omega$ a fixed free ultrafilter
on $\N=\{1,2,\dots\}$.
By $M$, we always denote a von Neumann algebra with separable predual.
The automorphism group of a von Neumann algebra
$N$ is denoted by $\Aut(N)$,
and the center of $N$ is by $Z(N)$.

Denote by $\ell^\infty(M)$ the unital C$^*$-algebra
which consists of
all norm bounded sequences $(x^\nu)=(x^1,x^2,\dots)$,
$x^\nu\in M$.
An element $(x^\nu)\in \ell^\infty(M)$
is said to be \emph{$\om$-trivial}
when $x^\nu$ converges to 0 in the strong$*$ topology
as $\nu\to\om$.
By $\mathscr{I}_\om(M)$, we denote the set of
all $\om$-trivial sequences.
It is known
that $\mathscr{I}_\om(M)$ is a C$^*$-subalgebra
of $\ell^\infty(M)$, but it is not an ideal
when $M$ is infinite.
Hence we consider its normalizer $\mathscr{M}^\om(M)$
defined by
\[
\mathscr{M}^\om(M)
:=
\{x\in \ell^\infty(M)
\mid
x\mathscr{I}_\om(M)+\mathscr{I}_\om(M)x
\subs
\mathscr{I}_\om(M)
\}.
\]
Then the quotient C$^*$-algebra
$M^\om:=\mathscr{M}^\om(M)/\mathscr{I}_\om(M)$
is in fact a W$^*$-algebra
that is called
an \emph{ultraproduct von Neumann algebra}.
We denote by $(x^\nu)^\om$ the equivalence class
$(x^\nu)+\sI_\om(M)$ for $(x^\nu)\in\sM^\om(M)$.

Note that $M$ is regarded as a von Neumann subalgebra of $M^\om$
by mapping $x\in M$ to its constant sequence
$(x,x,\dots)^\om=:x^\om$.
Since the norm unit ball of $M$ is $\si$-weakly compact,
each $(x^\nu)\in\ell^\infty(M)$
has the $\si$-weak ultralimit
$\lim_{\nu\to\om}x^\nu$.
This gives us a well-defined map $E_M\colon M^\om\to M$
defined by $E_M((x^\nu)):=\lim_{\nu\to\om}x^\nu$.
Then $E_M$ is actually a faithful normal conditional expectation.
For a weight $\vph$ on $M$,
we denote by $\vph^\om$ the ultraproduct weight of $\vph$ on $M^\om$,
that is, $\vph^\om:=\vph\circ E_M$.

An element $(x^\nu)\in\ell^\infty(M)$ is said to be
\emph{$\om$-central}
if $x^\nu\vph-\vph x^\nu\in M_*$ converges to 0 in norm
as $\nu\to\om$
for all $\vph\in M_*$,
where we use the usual notation
$a \vph(x):=\vph(xa)$ and $\vph a(x):=\vph(ax)$
for $a,x\in M$ and $\vph\in M_*$.
Then $\mathscr{C}_\om(M)$, the set of all $\om$-central sequences,
is a unital
C$^*$-subalgebra of $\ell^\infty(M)$ and contains $\mathscr{I}_\om(M)$.
We denote by $M_\om$ the quotient C$^*$-algebra
$\mathscr{C}_\om(M)/\mathscr{I}_\om(M)$ that is a W$^*$-subalgebra
of $M^\om$.
We will call $M_\om$ the \emph{asymptotic centralizer} of $M$.

\subsection{Action and crossed product}
Let $G$ be a locally compact Hausdorff group
that is always assumed to be second countable.
We use the usual notation $C_c(G)$ and $L^2(G)$
for the set of compactly supported continuous functions on $G$
and the Hilbert space associated with a fixed left invariant
Haar measure on $G$.
The $*$-algebra operations of $C_c(G)$
are defined as usual
$f*g(s):=\int_G f(t)g(t^{-1}s)\,dt$
and
$f^*(s):=\Delta(s)^{-1}\overline{f(s^{-1})}$
for $f,g\in C_c(G)$ and $s\in G$,
where $\Delta$ denotes the modular function of $G$
and $dt$ the left invariant Haar measure.

An action of $G$ on $M$ means a group homomorphism
$\alpha\colon G\ni s\mapsto \alpha_s\in \Aut(M)$
such that $\|\vph\circ\alpha_s-\vph\|_{M_*}\to0$
for all $\vph\in M_*$
if $s\to e$ in $G$,
where $e$ denotes the neutral element of $G$.
The fixed point algebra $M^\al$ means the collection of
all $x\in M$ such that $\al_s(x)=x$ for all $s\in G$.
We next introduce the crossed product von Neumann
algebra $M\rti_\alpha G$ as follows.
Suppose $M$ is acting on a Hilbert space $H$.
We define the operators $\pi_\alpha(x)$ and $\lambda^\alpha(s)$
on the tensor product Hilbert space
$H\otimes L^2(G)=L^2(G,H)$
as follows:
for
$x\in M$, $s,t\in G$
and
$\xi\in L^2(G,H)$,
\[
(\pi_\alpha(x)\xi)(t):=\alpha_{t^{-1}}(x)\xi(t),
\quad
(\la^\al(s)\xi)(t):=\xi(s^{-1}t).
\]

We can also write $\la^\al(s)=1\oti \la(s)$
for $s\in G$,
where $\la$ denotes the left regular representation on $L^2(G)$.
Then $M\rti_\alpha G$ denotes the von Neumann algebra
generated by $\pi_\al(M)$ and $\la^\al(G)$.
For $f\in C_c(G)$,
we denote $\la^\al(f):=\int_G f(s)\la^\al(s)\,ds$.
Note that $\la^\al$ is a $*$-representation of $C_c(G)$
on $H\oti L^2(G)$.

For abelian $G$,
$M\rti_\alpha G$ admits the dual action $\hal$ of the dual group
$\widehat{G}$
satisfying
\[
\hal_p(\pi_\alpha(x))=\pi_\alpha(x),
\quad
\hal_p(\la^\al(s))=\overline{\langle s,p\rangle}\la^\al(s),
\quad
x\in M,\ s\in G,\ p\in\widehat{G},
\]
where $\langle\cdot,\cdot\rangle$ denotes the dual coupling of
$G$ and $\widehat{G}$.
By Takesaki duality,
we have an isomorphism $\Ga_\al$
from $(M\rti_\al G)\rti_\hal\widehat{G}$
onto $M\oti B(L^2(G))$
such that
\[
\Ga_\al(\pi_\hal(\pi_\al(x)))=\pi_\al(x),
\quad
\Ga_\al(\pi_\hal(\la^\al(s)))=1\oti\la(s),
\quad
\Ga_\al(\la^\hal(p))=1\oti\ovl{\langle p,\cdot\rangle}
\]
for $x\in M$, $s\in G$ and $p\in\widehat{G}$.
As for the bidual action $\widehat{\widehat{\al}}$,
we have
$\Ga_\al\circ\widehat{\widehat{\al}}_s=\al_s\oti\Ad \rho(s)$
for $s\in G$,
where $\rho$ denotes the right regular representation
on $L^2(G)$.

\section{Main result}

\subsection{Equicontinuous parts}
Readers are referred to \cite[Chapter 3]{MT-Roh}.
Note that basic results introduced there are concerned with $\R$,
but they also hold for a general locally compact Hausdorff groups.

\begin{defn}
Let $\alpha$ be an action of a locally compact Hausdorff group
$G$ on a von Neumann algebra $M$.
A norm bounded sequence $(x^\nu)$,
$x^\nu\in M$,
is said to be $(\al,\om)$-\emph{equicontinuous}
when
the following holds:
for every $\si$-strong* neighbourhood $V$ of $0\in M$,
there exist a neighbourhood $U$ of the neutral element
$e\in G$
and $A\in \om$
such that
$\al_s(x^\nu)-x^\nu\in V$
for all
$s\in U$ and $\nu\in A$.
\end{defn}

Denote by $\mathscr{E}_\al^\om(M)$
the collection of all $(\al,\om)$-equicontinuous sequences.
Set
\[
M_\al^\om:=(\mathscr{E}_\al^\om(M)\cap \mathscr{M}^\om(M))/\mathscr{I}_\om(M)
\]
and
\[
M_{\om,\al}
:=(\mathscr{E}_\al^\om(M)\cap \mathscr{C}^\om(M))/\mathscr{I}_\om(M).
\]
which we will call the \emph{equicontinuous parts} of $M^\om$
and $M_\om$, respectively.
Note that $M_{\om,\al}\subs M_\al^\om$, $M\subs M_\al^\om$
and they are von Neumann subalgebras
which admit the $G$-action $\alpha^\om$
defined by
$\alpha_s^\om((x^\nu)^\om):=(\alpha_s(x^\nu))^\om$
for $s\in G$ and $(x^\nu)^\om\in M_\al^\om$.

Note the crucial fact that
$\sM^\om(M)$ coincides with
$\mathscr{E}_\al^\om(M)$
for $\al:=\si^\vph$, the modular automorphism group
of a given faithful normal state $\vph$ on $M$.
(See
\cite[Proposition 4.11]{AH}
and \cite[Theorem 1.5]{MT-discrete} for its proof.)

A useful tool to construct an equicontinuous sequence
is to average a norm bounded sequence by $L^1$-function.
To be precise,
we let $f\in L^1(G)$ and $(x^\nu)\in\ell^\infty(M)$.
Then $(\alpha_f(x^\nu))$ is $(\al,\om)$-equicontinuous,
where $\alpha_f(y)=\int_G f(s)\alpha_s(y)\,ds$
for $y\in M$.
Note that the averaging and the ultraproduct of an equicontinuous sequence
are commutative operations,
that is,
for $x:=(x^\nu)^\om\in M_\al^\om$,
we have $\al_f^\om(x)=(\al_f(x^\nu))^\om$.
In particular,
the set which consists of $(\al_f(x^\nu))^\om$,
$f\in L^1(G)$ and $(x^\nu)\in \mathscr{M}^\om(M)$
is $\si$-weakly dense in $M_\al^\om$.

\begin{exam}\label{exam:G}
Consider the action $\alpha$
of $G$ on $M:=L^\infty(G)$ by left translation.
Then $M_\al^\om$ is actually nothing but $M$.
This fact tells us that equicontinuous parts could be small.
We need the following claim to show this:
if a uniformly norm bounded net $\{f_n\}_{n\in I}$ in $M$
converges to 0 in the $\sigma$-weak topology,
then the convolution $g*f_n$ converges to 0
compact uniformly for all $g\in L^1(G)$.
Then for $t\in G$
\[
g*f_n(t)=\int_G g(ts)f_n(s^{-1})\,ds
=
\langle g_{t^{-1}},\tilde{f_n}\rangle,
\]
where $g_r(s):=g(r^{-1}s)$, $\tilde{f_n}(s):=f_n(s^{-1})$,
and $\langle \cdot,\cdot\rangle$ denotes the pairing
of $L^1(G)$ and $L^\infty(G)$.
It is trivial that $\tilde{f_n}\to0$ $\si$-weakly,
and $g*f_n$ converges to 0 pointwise.

Let $K\subset G$ be a compact set.
The map $K\ni t\mapsto g_{t^{-1}}\in L^1(G)$ is norm-continuous.
Thus for $\vep>0$, there exist $t_1,\dots,t_k\in K$
such that for any $t\in K$, $\|g_{t^{-1}}-g_{t_i^{-1}}\|<\vep$
for some $t_i$.
Take $n_0\in I$ so that $|\langle g_{t_i^{-1}},\tilde{f_n}\rangle|<\vep$
for all $i=1,\dots,k$ and $n\geq n_0$.
For $t\in K$, take $t_i$ so that $\|g_{t^{-1}}-g_{t_i^{-1}}\|<\vep$.
When $n\geq n_0$, we have
\begin{align*}
|g*f_n(t)|
&\leq |\langle g_{t^{-1}}-g_{t_i^{-1}},\tilde{f_n}\rangle|
+
|\langle g_{t_i^{-1}},\tilde{f_n}\rangle|
\\
&\leq
\vep\|f_n\|_\infty
+
|\langle g_{t_i^{-1}},\tilde{f_n}\rangle|
<
(\|f_n\|_\infty+1)\vep.
\end{align*}
So, we have proved the claim.

Recall that $M_\al^\om$ has
the $\si$-weakly total set which consists of $(\al_g(f^\nu))$,
$g\in L^1(G)$ and $(f^\nu)\in \ell^\infty(M)$.
Putting $f:=\lim_{\nu\to\om}f^\nu$,
we see $\al_g(f^\nu)=g*f^\nu\to g*f$ compact uniformly.
In particular, $(\al_g(f^\nu))^\om$ equals
the constant sequence $\al_g(f)^\om$.
\end{exam}

\begin{lem}
\label{lem:Eal}
There exists a unique faithful normal conditional expectation
$E_\al$ from $M^\om$ onto $M_\al^\om$
such that
for an arbitrary faithful normal semifinite weight $\vph$ on $M$,
one has
$\vph^\om=\vph^\om\circ E_\al$.
\end{lem}
\begin{proof}
Since $M_\al^\om$ contains $M$,
$\vph^\om$ is semifinite on $M_\al^\om$.
We will show that $M_\al^\om$ is globally invariant
by $\sigma^{\vph^\om}$.
By \cite[Theorem 4.1]{AH}, \cite[Proposition 2.1]{K}
or \cite[Theorem 2.1]{R},
we obtain $\sigma_t^{\vph^\om}((x^\nu)^\om)=(\sigma_t^\vph(x^\nu))^\om$
for $(x^\nu)^\om\in M^\om$.
Then for $t\in\R$, $s\in G$, $(x^\nu)^\om\in M_\al^\om$ and $\nu\in\N$,
we have
\[
\al_s(\sigma_t^\vph(x^\nu))
=
\sigma_t^{\vph\circ\al_{s^{-1}}}(\al_s(x^\nu))
=
[D\vph\circ\al_{s^{-1}}:D\vph]_t
\sigma_t^\vph(\alpha_s(x^\nu))
[D\vph\circ\al_{s^{-1}}:D\vph]_t^*.
\]
This implies that $(\sigma_t^\vph(x^\nu))$ is $(\al,\om)$-equicontinuous
for each $t\in\R$
since $(x^\nu)$ is an element of $\mathscr{M}^\om(M)$.
(See \cite[Lemma 3.6]{MT-Roh}.)
Hence $M_\al^\om$ is globally invariant by $\si^{\vph^\om}$.
Thanks to Takesaki's criterion \cite[p.309]{Ta-ex},
we can take a faithful normal conditional expectation
$E_\al$ from $M^\om$ onto $M_\al^\om$
so that $\vph^\om=\vph^\om\circ E_\al$.
This equality implies that $E_M=E_M\circ E_\al$,
and $E_\al$ is unique.
\end{proof}

\subsection{Main results}

The canonical embedding $\pi_\al$ of $M$ into $M\rti_\al G$
induces
$\pi_\al^\infty\col \sE_\al^\om(M)\cap\sM^\om(M)\to \ell^\infty(M\rti_\al G)$
by putting $\pi_\al^\infty((x^\nu)):=(\pi_\al(x^\nu))$.

\begin{lem}
\label{lem:pial}
If $(x^\nu)\in \sE_\al^\om(M)\cap\sM^\om(M)$,
then $\pi_\al^\infty((x^\nu))\in \sM^\om(M\rti_\al G)$.
\end{lem}
\begin{proof}
Let $(y^\nu)$ be an $\om$-trivial sequence in $M\rti_\al G$
with $\|y^\nu\|\leq1$ for all $\nu$.
It suffices to show that
$\|y^\nu \pi_\al(x^\nu)(\xi\oti f)\|\to0$
as $\nu\to\om$
for $\xi\in H$ and $f\in C_c(G)$
with compact support $K\subs G$.

Let $\vep>0$.
Since $(x^\nu)$ is $(\al,\om)$-equicontinuous,
we can take $W\in \om$
and a open neighborhood
$V$ of $e\in G$
so that if $t^{-1}s\in V$, $s,t\in K$
and $\nu\in W$,
then
$\|\al_{s^{-1}}(x^\nu)\xi-\al_{t^{-1}}(x^\nu)\xi\|<\vep$.

Take $s_1,\dots,s_N\in K$ so that
$K\subs s_1 V\cup\cdots\cup s_N V$.
Let $\{h_1,\dots,h_N\}$ be a partition of unity on $K$
subordinate to the cover $\{s_1 V,\dots,s_N V\}$.
(See \cite[Theorem 2.13]{Ru}.)
Then for $\nu\in W$,
we obtain the following:
\begin{align*}
&\left
\|\Big{(}
\pi_\al(x^\nu)-\sum_{j=1}^{N}(\al_{s_j^{-1}}(x^\nu)\oti h_j)
\Big{)}(\xi\oti f)\right\|^2
\\
&=
\int_{K}
\Big{\|}
\Big{(}
\al_{s^{-1}}(x^\nu)-\sum_{j=1}^N
\al_{s_j^{-1}}(x^\nu)h_j(s)
\Big{)}\xi
\Big{\|}^2
|f(s)|^2\,ds
\\
&=
\int_{K}
\Big{\|}
\sum_{j=1}^N
h_j(s)
(\al_{s^{-1}}(x^\nu)-
\al_{s_j^{-1}}(x^\nu))
\xi
\Big{\|}^2
|f(s)|^2\,ds
\\
&\leq
\int_{K}
\Big{(}
\sum_{j=1}^N
h_j(s)
\|
\al_{s^{-1}}(x^\nu)-
\al_{s_j^{-1}}(x^\nu)
\xi\|
\Big{)}^2
|f(s)|^2\,ds
\\
&\leq
\vep^2\|f\|_2^2.
\end{align*}

Thus for all $\nu\in W$,
we have
\[
\|y^\nu \pi_\al(x^\nu)(\xi\oti f)\|
\leq
\vep\|f\|_2
+
\left\|y^\nu \sum_{j=0}^{N-1}(\al_{s_j^{-1}}(x^\nu)\oti h_j)(\xi\oti f)
\right\|.
\]
In the last term,
we know that $(\al_{s_j^{-1}}(x^\nu)\oti 1)$
belongs to $\sM^\om(M\oti B(L^2(G)))$
by the proof of \cite[Lemma 2.8]{MT-Roh}.
In particular,
the last term converges to 0
in the strong topology
as $\nu\to\om$.
Hence the above inequality implies that
\[
\lim_{\nu\to\om}\|y^\nu \pi_\al(x^\nu)(\xi\oti f)\|
\leq
\vep\|f\|_2.
\]
Thus we are done.
\end{proof}

The map $\pi_\al^\infty$ induces
a well-defined map $\pi_\al^\om$
from $M_\al^\om$ into $(M\rti_\al G)^\om$
such that $\pi_\al^\om((x^\nu)^\om):=(\pi_\al(x^\nu))^\om$
for $(x^\nu)^\om\in M_\al^\om$.
In the proof of Lemma \ref{lem:pial},
we have shown $\pi_\al^\om$ is actually
a map from $M_\al^\om$ into $(M\otimes B(L^2(G)))^\om$.
Recall the isomorphism $\Psi$ from
$(M\oti B(L^2(G)))^\om$ onto $M^\om\oti B(L^2(G))$
that is given in the proof of \cite[Lemma2.8]{MT-Roh}.
Note that the map $\Psi$ is naturally defined
so that
for $f,g\in L^2(G)$ and $x=(x^\nu)^\om\in (M\oti B(L^2(G)))^\om$,
we have
$(\id\oti \phi_{f,g})(\Psi(x))=((\id\oti\phi_{f,g})(x^\nu))^\om$,
where $\phi_{f,g}$ denotes the normal functional
$\langle \cdot f,g\rangle$ on $B(L^2(G))$.

\begin{lem}
\label{lem:psial}
For any $x\in M_\al^\om$,
one has $\Psi(\pi_\al^\om(x))=\pi_{\al^\om}(x)$.
\end{lem}
\begin{proof}
Let $f,g\in L^2(G)$ and $x=(x^\nu)^\om\in M_\al^\om$.
On the one hand,
we have
\[
(\id\oti \phi_{f,g})(\Psi(\pi_\al^\om(x)))
=
\big{(}
(\id\oti\phi_{f,g})(\pi_\al(x^\nu))
\big{)}^\om.
\]
On the other hand,
using the equicontinuity
(cf. \cite[Lemma 3.15]{MT-Roh}),
we have
\begin{align*}
(\id\oti \phi_{f,g})(\pi_{\al^\om}(x))
&=
\int_G f(s)\ovl{g(s)}\al_{s^{-1}}^\om(x)\,ds
=
\Big{(}
\int_G f(s)\ovl{g(s)}\al_{s^{-1}}(x^\nu)\,ds
\Big{)}^\om
\\
&=
\big{(}
(\id\oti\phi_{f,g})(\pi_\al(x^\nu))
\big{)}^\om.
\end{align*}
Thus we are done.
\end{proof}

We now prove the main result of this paper
which strengthens \cite[Theorem 1.10]{MT-discrete}.
Note that a generalization of Example \ref{exam:G}
to the crossed product for $G$ being abelian.

\begin{thm}
\label{thm:cross}
Let $\al$ be an action of
a second countable locally compact Hausdorff group $G$
on a von Neumann algebra $M$ with separable predual.
Then the following statements hold:
\begin{enumerate}
\item 
There exists a canonical embedding
$\Phi_\al$
of $M_\al^\om\rti_{\al^\om}G$
into $(M\rti_\al G)^\om$
such that
$\Phi_\al(\pi_{\al^\om}(x))
=
\pi_\al^\om(x)$
and
$\Phi_\al(\la^{\al^\om}(s))=\la^\al(s)^\om$,
respectively,
for all $x\in M_\al^\om$
and $s\in G$.

\item
If $G$ is abelian,
the map $\Phi_\al$
induces the isomorphism
from $M_\al^\om\rti_{\al^\om}G$
onto $(M\rti_\al G)_{\hal}^\om$.
\end{enumerate}
\end{thm}
\begin{proof}
(1).
Put $N:=\pi_\al^\om(M_\al^\om)\vee \{\la^\al(t)^\om\mid t\in G\}''$
that is a von Neumann subalgebra of $(M\rti_\al G)_{\hal}^\om$.
We will show that there exists a canonical isomorphism
from $M_\al^\om\rti_{\al^\om}G$ onto $N$.

Let $\vph$ be a faithful normal semifinite weight on $M$
and $\ps$ the dual weight of $\vph$ on $M\rti_\al G$.
It is obvious that $\ps^\om$ is semifinite on $N$
since $N$ contains $M\rtimes_\alpha G$.
Then for $(x^\nu)^\om\in M_\al^\om$, $s\in G$ and $t\in \R$,
we have
\[
\si_t^{\ps^\om}(\la^\al(s)^\om)
=
(\si_t^\ps(\la^\al(s)))^\om
=
\Delta_G(s)^{it}\la^\al(s)^\om
\pi_\al([D\vph\circ\al_s:D\vph]_t)^\om
\]
and
\begin{align*}
\si_t^{\ps^\om}(\pi_\al^\om((x^\nu)^\om))
&=
\si_t^{\ps^\om}((\pi_\al(x^\nu))^\om)
=
(\si_t^\ps(\pi_\al(x^\nu)))^\om
\\
&=
(\pi_\al(\si_t^\vph(x^\nu)))^\om
=
\pi_\al^\om((\si_t^\vph(x^\nu))^\om),
\end{align*}
where we note that
the last term is well-defined from
the proof of Lemma \ref{lem:Eal}.
This observation implies
$N$ is globally invariant under $\si^{\ps^\om}$.
Thanks to Takesaki's theorem,
we can take a faithful normal conditional expectation
from $(M\rti_\al G)^\om$ onto $N$.
In particular, the restriction of the modular conjugation $J_{\ps^\om}$
on $L^2(N,\ps^\om)$ gives the modular conjugation associated with
$\ps^\om\!\upharpoonright_{N}$.

Let $\chi$ be the dual weight of $\vph^\om\!\upharpoonright_{M_\al^\om}$
on $P:=M_\al^\om\rti_{\al^\om} G$.
We will compare the GNS Hilbert spaces $L^2(P,\chi)$
and $L^2(N,\ps^\om)$.
The definition left ideals are as usual
denoted by $n_\chi$ and $n_{\ps^\om}$, respectively.
(See \cite[Lemma VII.1.2]{TaII}.)
Denote by $\La_\chi\col n_\chi\to L^2(P,\chi)$
and $\La_{\ps^\om}\col n_{\ps^\om}\to L^2(N,\ps^\om)$
the canonical embeddings.

Let us introduce a map $V$ which maps
$\La_\chi(\la^{\al^\om}(f)\pi_{\al^\om}(x))$
to
$\La_{\ps^\om}(\la^\al(f)^\om\pi_\al^\om(x))$
for $f\in C_c(G)$ and $x\in M_\al^\om$.
We claim that $V$ extends to an isometry from $L^2(P,\chi)$
into $L^2(N,\ps^\om)$.
Take $f,g\in C_c(G)$ and $x,y\in M_\al^\om$.
Then we have
\begin{align*}
\langle\La_{\ps^\om}(\la^\al(f)^\om\pi_\al^\om(x)),
\La_{\ps^\om}(\la^\al(g)^\om\pi_\al^\om(y))
\rangle
&=
\ps^\om\big{(}
\pi_\al^\om(y^*)\la^\al(g^* *f)^\om\pi_\al^\om(x)
\big{)}
\\
&=
\ps\big{(}
\lim_{\nu\to\om}
\pi_\al((y^\nu)^*)
\la^\al(h)
\pi_\al(x^\nu)
\big{)},
\end{align*}
where $h:=g^* *f$.
Let $F$ be the support of $h$.
Then for each $\nu\in\N$, we have
\[
\pi_\al((y^\nu)^*)
\la^\al(h)
\pi_\al(x^\nu)
=
\int_F
h(s)
\pi_\al((y^\nu)^* \al_s(x^\nu))\la^\al(s)
\,ds.
\]
Using \cite[Lemma 3.3]{MT-Roh},
we know that
\[
\lim_{\nu\to\om}
\pi_\al((y^\nu)^*)
\la^\al(h)
\pi_\al(x^\nu)
=
\int_F
h(s)
\pi_\al(\lim_{\nu\to\om}y^\nu \al_s(x^\nu))\la^\al(s)
\,ds.
\]
Hence it follows from the definition of the dual weight $\psi$
the following:
\begin{align*}
\langle\La_{\ps^\om}(\la^\al(f)^\om\pi_\al^\om(x)),
\La_{\ps^\om}(\la^\al(g)^\om\pi_\al^\om(y))
\rangle
&=
h(e)
\vph(\lim_{\nu\to\om}(y^\nu)^*x^\nu)
\\
&=
h(e)\vph^\om(y^* x),
\end{align*}
which equals to
$\langle
\La_{\chi}(\la^{\al^\om}(f)\pi_{\al^\om}(x)),
\La_{\chi}(\la^{\al^\om}(g)\pi_{\al^\om}(y))
\rangle$ again by the definition of the dual weight $\chi$.
Thus we have proved the existence of the isometry $V$.

We next claim that $K$, the image of $V$, is $N'$-invariant.
For a $\si^{\vph^\om}$-analytic $y\in M_\al^\om$
and $t\in G$,
we obtain the followings
for all $f\in C_c(G)$ and $x\in M_\al^\om$:
\[
J_{\ps^\om}\si_{i/2}^{\ps^\om}(\pi_\al^\om(y))^*J_{\ps^\om}
\La_{\ps^\om}(\la^\al(f)^\om\pi_\al^\om(x))
=
\La_{\ps^\om}(\la^\al(f)^\om\pi_\al^\om(xy))\in K,
\]
and
\[
J_{\ps^\om}\si_{i/2}^{\ps^\om}(\la^\al(t)^\om)^*J_{\ps^\om}
\La_{\ps^\om}(\la^\al(f)^\om\pi_\al^\om(x))
=
\La_{\ps^\om}(\la^\al(g)^\om\pi_\al^\om(\al_{t^{-1}}^\om(x)))
\in K,
\]
where $g(s):=\Delta_G(t)^{-1}f(st^{-1})$.
Hence $K$ is $N'$-invariant.

Now let us take a $\si^{\ps^\om}$-analytic $y\in n_{\ps^\om}$.
Then
\[
J_{\ps^\om}\si_{i/2}^{\ps^\om}(y)^*J_{\ps^\om}
\La_{\ps^\om}(\la^\al(f)^\om\pi_\al^\om(x))
=
\la^\al(f)^\om\pi_\al^\om(x)\La_{\ps^\om}(y).
\]
This implies that
$\La_{\ps^\om}(y)$ is contained in the closure of $N'K$.
Since $N'K\subs K$,
$\La_{\ps^\om}(y)$ belongs to $K$.
Thus $K=L^2(N,\ps^\om)\subset L^2((M\rti_\al G)^\om,\ps^\om)$.
Then the map $N\ni x\mapsto V^*xV$
provides us with the isomorphism from $N$ onto
$M_\al^\om\rti_{\al^\om}G$.
More precisely,
we can check
$V^*\pi_\al^\om(x)V=\pi_{\al^\om}(x)$
and
$V^*\la^\al(s)^\om V=\la^{\al^\om}(s)$
for $x\in M_\al^\om$ and $s\in G$.
Denote by $\Phi_\alpha$ its inverse map.

(2).
Suppose that $G$ is abelian.
Then the image of $\Phi_\alpha$ is clearly contained in
$(M\rti_\al G)_{\hal}^\om$.
Let us apply the statement of (1) to
the dual action $\be:=\hal$ on $R:=M\rti_\al G$.
Then we have the embedding
$\Phi_{\be}$ of $R_\be^\om\rti_{\be^\om}\widehat{G}$
into
$(R\rti_\be \widehat{G})_{\hbe}^\om$.

Recall the isomorphism $\Psi$
from
$(M\oti B(L^2(G)))^\om$
onto
$M^\om\oti B(L^2(G))$
that is introduced in the remark before Lemma \ref{lem:psial}.
Then $\Psi$ induces
an isomorphism from
$(M\oti B(L^2(G)))_{\al\oti\Ad\rho}^\om$
onto
$M_\al^\om\oti B(L^2(G))$.
This fact can be directly proved or deduced from \cite[Lemma 3.12]{MT-Roh}
for general groups.
In summary, we have the following diagram:
\[
\xymatrix{
R_\be^\om\rti_{\be^\om}\widehat{G}\,
\ar@{^{(}->}[r]^{\Phi_\beta}
&
(R\rti_\be \widehat{G})_{\hbe}^\om
\ar[r]^(0.39){(\Ga_\al)^\om}
&
(M\oti B(L^2(G)))_{\al\oti\Ad\rho}^\om
\ar[d]^{\Psi}
\\
P\rti_{\widehat{\al^\om}}\widehat{G}
\ar@{^{(}->}[u]^{\Phi_\al\oti\id}
&
M_\al^\om\oti B(L^2(G))
\ar[l]_(0.57){(\Gamma_{\al^\om})^{-1}}
\ar@{.>}[r]^{f}
&
M_\al^\om\oti B(L^2(G))
}
\]
where
$(\Ga_\al)^\om$ is defined by
$(\Gamma_\al)^\om((x^\nu)^\om):=(\Gamma_\al(x^\nu))^\om$
for $(x^\nu)^\om\in (R\rti_\be \widehat{G})^\om$
and $f$ denotes the composition
of all of them.

We will show $f$ actually equals the identity map.
This implies the surjectivity of $\Phi_\al\oti\id$
in the diagram above,
and we obtain $\Phi_\al(P)=R_\be^\om$
by taking the fixed point algebra
of the dual action of $\be^\om$.
Recall that $M_\al^\om\oti B(L^2(G))$
is generated by $\pi_{\al^\om}(M_\al^\om)$,
$1\oti \la(G)$ and $1\oti L^\infty(G)$.
We can directly check that $f$ identically maps
$1\oti \la(G)$ and $1\oti L^\infty(G)$.
For $x\in M_\al^\om$,
it is not difficult to show $\pi_{\al^\om}(x)$
is mapped to $\pi_\al^\om(x)$
in $(M\oti B(L^2(G)))_{\al\oti\Ad\rho}^\om$,
and it turns out from Lemma \ref{lem:psial}
that
$f(\pi_{\al^\om}(x))=\pi_{\al^\om}(x)$.
\end{proof}

It would be interesting to generalize the previous theorem
to a general locally compact Hausdorff group or quantum group
by introducing the equicontinuity of their actions.

\section{Applications}

\subsection{Continuous or discrete crossed product decomposition of $M^\om$}

Let $M=N\rti_\th\R$ be the continuous crossed product decomposition
of a properly infinite von Neumann algebra $M$,
that is,
$N$ is a semifinite von Neumann algebra
that is
endowed with the $\R$-action $\th$
and 
a faithful normal tracial weight $\tau$
satisfying $\tau\circ\theta_s=e^{-s}\tau$ for $s\in\R$.
Let $\vph$ be the dual weight of $\tau$.
Since the dual action $\widehat{\theta}$ is nothing but
the modular automorphism $\si^\vph$,
the following result follows
from Theorem \ref{thm:cross},
\cite[Proposition 4.11]{AH}
and \cite[Theorem 1.5]{MT-discrete}.

\begin{thm}\label{thm:contdec}
Let $M=N\rti_\th\R$ be the continuous crossed product decomposition
of a properly infinite von Neumann algebra $M$.
Then the continuous crossed product decomposition of $M^\om$ is given by
$M^\om=N_\th^\om\rti_{\th^\om}\R$.
In particular, the flow of weights of $M^\om$
is given by the restriction of $\th^\om$ on $Z(N_\th^\om)$.
\end{thm}

The following result on a discrete crossed product decomposition
is proved first by Ando--Haagerup in \cite{AH}.
We will present another proof using Theorem \ref{thm:cross}.

\begin{thm}[Ando--Haagerup]
Let $M$ be a type III$_\lambda$ factor with
$0\leq \la<1$.
Let $M=N\rti_\th\Z$ be the discrete crossed product decomposition.
Then the discrete crossed product decomposition of $M^\om$
is given by
$M^\om=N^\om\rti_{\th^\om}\Z$.
In particular,
if $0<\la<1$,
then $M^\om$ is a type III$_\la$ factor.
\end{thm}
\begin{proof}
It turns out from Theorem \ref{thm:cross}
that
$M_{\hat{\th}}^\om=N^\om\rti_\th\Z$.
Hence it suffices to show that
$M_{\hat{\th}}^\om=M^\om$.
For $\la\neq0$,
$\hat{\th}$ is nothing but the modular automorphism
$\si^{\hat{\tau}}$,
where $\tau$ denotes a faithful normal tracial
weight on $N$ with $\tau\circ\th=\la\tau$.
Hence we are done.

Suppose next that $\la=0$.
Take a faithful normal tracial weight $\tau$ on $N$
such that $\tau\circ \theta\leq \mu\tau$
with $0<\mu<1$.
Let $H_n$ be the selfadjoint operator
affiliated with $Z(N)$
such that
$\tau\circ\theta^n=\tau_{\exp(H_n)}$ for $n\in\Z$.
Then the spectrum of $H_n$ is contained in
$(-\infty,n\log\mu]$ and $[n\log\mu^{-1},\infty)$
when $n\geq1$ and $n\leq-1$, respectively.

Let $\vph:=\hat{\tau}$
and
$g_\be(t):=\sqrt{\beta/\pi}\exp(-\be t^2)$
for $\be>0$ and $t\in\R$
and
$U_\be:
=
\widehat{g_\be}(-\log\De_\vph)
=
\int_\R g_\be(t)\De_\vph^{it}\,dt
$,
where $\widehat{g_\be}(p):=\int_\R g_\be(t)e^{-ipt}\,dt
=\exp(-p^2/(4\be))$,
$p\in\R$.
Then $U_\be\to1$ in the strong topology as $\be\to\infty$.

Now we will show $M^\om=M_{\hat{\th}}^\om$.
Take $x=(x^\nu)^\om\in M^\om$.
It suffices to show that
$\sigma_{g_\be}^{\vph^\om}(x)$ is contained in $M_{\hat{\th}}$
since $\sigma_{g_\be}^{\vph^\om}(x)$ converges to $x$
as $\be\to\infty$ in the strong$*$ topology.
Note that
$\sigma_{g_\be}^{\vph^\om}(x)=(\sigma_{g_\be}^\vph(x^\nu))^\om$.

Let $y=\sum_{m\in\Z}y(m)\la^\th(m)$
with $y(m)\in N$
be the formal decomposition of $y\in M$.
Then we have
\[
\sigma_{g_\be}^\vph(y)
=\sum_{m\in\Z}y(m)\la^\th(m)\widehat{g_\be}(-H_m),
\]
where the series in the right-hand side converges
in the norm topology
since
$\|\widehat{g_\be}(-H_m)\|_\infty\leq\exp(-m^2|\log\mu|^2/(4\be))$
for all $m\in\Z$.
Fix $k\in\N$.
Then
\[
\Big{\|}
\sigma_{g_\be}^\vph(y)
-
\sum_{|m|\leq k}y(m)\la^\th(m)\widehat{g_\be}(-H_m)
\Big{\|}
\leq
\|y\|\sum_{|m|>k}\exp(-m^2|\log\mu|^2/(4\be)).
\]
Hence for $x=(x^\nu)^\om\in M^\om$ and $\be>0$,
we have
\begin{align*}
\Big{\|}
\sigma_{g_\be}^{\vph^\om}(x)
-
\sum_{|m|\leq k}
(x^\nu(m)\la^\th(m)\widehat{g_\be}(-H_m))^\om
\Big{\|}
\leq
\|x\|\sum_{|m|>k}\exp(-m^2|\log\mu|^2/(4\be)).
\end{align*}
It is clear that
$(x^\nu(m)\la^\th(m)\widehat{g_\be}(-H_m))^\om$ is contained in
$M_{\hat{\theta}}^\om$,
and so is $\sigma_{g_\be}^{\vph^\om}(x)$.
\end{proof}

Thanks to \cite[Theorem 6.11]{AH},
we know $M^\om$ is actually a type III$_1$ factor
when $M$ is.
Hence $N_\th^\om$ is a type II$_\infty$ factor in this situation,
but we have not found another proof of this fact which uses
Theorem \ref{thm:contdec}.

\subsection{Description of $M_\omega$ and fullness of $M$}

Let $M$ be an infinite type III factor with separable predual
and
$M=N\rti_\theta\R$ be the continuous crossed product decomposition of $M$
as before. Then the following result holds.

\begin{lem}
The asymptotic centralizer
$M_\omega$ is isomorphic to $(N_{\om,\theta})^{\th^\om}$.
\end{lem}
\begin{proof}
Let $\tau$ be a faithful normal tracial weight on $N$
satisfying $\tau\circ\theta_s=e^{-s}\tau$ for $s\in\R$
and $\vph$ the dual weight of $\tau$.
Then by Theorem \ref{thm:contdec},
we have $M^\om=N_\th^\om\rti_{\th^\om}\R$.
We will compute $(M'\cap M^\om)^{\si^{\vph^\om}}$
which equals $M_\om$.
(Use \cite[Proposition 4.35]{AH} and the Connes cocycle derivative.)
Using $\la^\theta(t)\in M$, $t\in\R$,
we have
\begin{equation}
\label{eq:relcom}
M'\cap M^\om
\subset \pi_{\theta^\om}((N_\theta^\om)^{\theta^\om})
\vee
\{\la^{\theta^\om}(t)\mid t\in\R\}''.
\end{equation}
This implies that
$(M'\cap M^\om)^{\si^{\vph^\om}}
\subset \pi_{\theta^\om}(N'\cap(N_\theta^\om)^{\theta^\om})$.
Since the converse inclusion trivially holds
and $N'\cap N^\om=N_\om$,
we obtain
$M_\om=\pi_{\theta^\om}((N_{\om,\theta})^{\theta^\om})$.
\end{proof}

A separable factor $M$ is said to be
\emph{full} when $M_\om=\C$.
The fullness of $M$
has been studied by several researchers
in terms of the continuous core.
See references cited in \cite{Ma,TU}.
Also see \cite{HMV}
for recent progress in study of fullness.
Among them, Marrakchi in \cite{Ma}
shows that
$N$ is full if and only if
$M$ is a full type III$_1$ factor with $\tau$-invariant
being the usual topology of $\R$.
The following theorem would suggest that
the $\tau$-invariant could measure how continuously $\th^\om$
is acting on $N_\om$.
Our proof is similar to that of \cite[Lemma 3]{TU}.

\begin{thm}
Let $M$ be a type III$_1$ factor
with the continuous crossed product decomposition $M=N\rti_\th\R$ as before.
Then $M$ is full if and only if
$N_{\om,\th}=\C$.
\end{thm}
\begin{proof}
The ``if'' part is trivial from the previous lemma
or \cite[Corollary 1.9]{MT-discrete}.
Suppose that $M$ is full.
Let $p\in\R$ be an element of
the Arveson spectrum of $\theta^\om$ on $N_{\om,\th}$.
By \cite[Theorem 7.7]{MT-Roh},
we can take a unitary $u\in N_{\om,\th}$
such that $\th_t^\om(u)=e^{ipt}u$ for all $t\in\R$.
Let $(u^\nu)\in\ell^\infty(N)$
be a unitary representing sequence of $u$.
Then $\Ad \pi_\th(u^\nu)$ converges to $\hat{\th}_p$
in $\Aut(M)$.
The fullness of $M$ implies
the innerness of $\hat{\th}_p$,
and it turns out that $p=0$.
Namely, $\theta^\om$ is trivial on $N_{\om,\th}$,
and
the previous lemma implies $N_{\om,\th}=\C$.
\end{proof}

\begin{rem}
Ueda's problem asks if
$M'\cap M^\om=\C$ holds for any full factor $M$.
This is affirmatively solved by Ando--Haagerup in \cite[Theorem 5.2]{AH}.
We would like to deduce this result by strengthening (\ref{eq:relcom}),
but this approach has not been successful yet.
Instead, let us present a short proof of the problem.
Put $R:=M'\cap M^\om$.
Suppose $R$ were non-trivial.
Let $\vph$ be a faithful normal state on $M$.
Since $R_{\vph^\om}=M_\om=\C$,
$R$ would be a type III$_1$ factor.
We claim that for any $\vep>0$,
there exists $\de>0$
such that
if $x\in R$ with $\|x\|\leq1$
satisfies
$\|x\vph^\om-\vph^\om x\|_{(M^\om)_*}<\de$,
then $\|x-\vph^\om(x)\|_{\vph^\om}<\vep$.
By usual diagonal argument,
we can show this claim by contradiction.
Readers are referred to \cite[Chapter 5]{Oc}
for this.
Also note that
$\|x\vph^\om-\vph^\om x\|_{(M^\om)_*}
=\lim_{\nu\to\om}\|x^\nu\vph-\vph x^\nu\|$
for all $x=(x^\nu)^\om\in M^\om$.
For proof of this fact,
see \cite[Lemma 4.36]{AH} or \cite[Lemma 9.3]{MT-Roh}.
However,
since $R$ is a type III$_1$ factor,
there exist many non-trivial norm bounded sequences
$(y^k)\in \ell^\infty(R)$
such that $\|y^k\vph^\om-\vph^\om y^k\|\to0$
as $k\to\infty$, and we have a contradiction.
The last claim is implied by the fact that
$(R^\om)_{\psi^\om}$ is a type II$_1$ factor,
where $\psi:=\vph^\om$ \cite[Proposition 4.24]{AH}.
\end{rem}

\end{document}